\documentclass[11pt,oneside]{amsart}%
\usepackage{amssymb,amsmath,amsfonts,latexsym,amsthm,geometry,graphicx}
\usepackage{amsmath}
\usepackage{amsfonts}
\usepackage{color, soul}
\usepackage{amssymb}
\usepackage{graphicx}
\usepackage[all]{xy}%
\providecommand{\U}[1]{\protect\rule{.1in}{.1in}}
\newtheorem{theorem}{Theorem}[section]
\newtheorem{corollary}[theorem]{Corollary}
\newtheorem{proposition}[theorem]{Proposition}
\newtheorem{lemma}[theorem]{Lemma}

\theoremstyle{definition}

\newtheorem{definition}[theorem]{Definition}
\begin{document}
\title{Estimating the index of summability of pairs of Banach spaces}
\author[Mariana Maia]{M. Maia}
\address{Departamento de Matem\'{a}tica\\
Universidade Federal da Para\'{\i}ba \\
58.051-900 - Jo\~{a}o Pessoa, Brazil.}
\email{mariana.britomaia@gmail.com }
\author[Joedson Santos]{J. Santos}
\address{Departamento de Matem\'{a}tica \\
Universidade Federal da Para\'{\i}ba \\
58.051-900 - Jo\~{a}o Pessoa, Brazil.}
\email{joedsonmat@gmail.com}
\thanks{Mariana Maia is supported by Capes and Joedson Santos is supported by CNPq}
\maketitle

\begin{abstract}
Let $E_{1},...,E_{m},F$ be Banach spaces. The index of summability of $\left(
E_{1}\times\cdots\times E_{m},F\right)  $ is a kind of measure of how far the
$m$-linear operators $T:E_{1}\times\cdots\times E_{m}\rightarrow F$ are from
being multiple summing. In this paper we estimate the index of summability of
several pairs of Banach spaces$.$

\end{abstract}


\section{Introduction and preliminary results}

For $1\leq q\leq p<\infty$ and Banach spaces $E_{1},...,E_{m},F$ over
$\mathbb{K}=\mathbb{R}$ or $\mathbb{C},$ let us denote $E_{j}^{\ast}$ the
topological dual of $E_{j}$ and $B_{E_{j}^{\ast}}$ the closed unit ball of
$E_{j}^{\ast}$. We recall that a continuous multilinear operator
$T:E_{1}\times\cdots\times E_{m}\rightarrow F$ is multiple $(p,q)$-summing
(see \cite{matos, perez}) if there is a constant $C\geq0$ such that
\begin{equation}
\left(  \sum_{k_{1},...,k_{m}=1}^{n}\left\Vert T\left(  x_{k_{1}}%
^{(1)},...,x_{k_{m}}^{(m)}\right)  \right\Vert ^{p}\right)  ^{\frac{1}{p}}\leq
C\prod_{i=1}^{m}\left(  \sup_{\varphi_{i}\in B_{E_{i}^{\ast}}}\sum
\limits_{k_{i}=1}^{n}\left\vert \varphi_{i}\left(  x_{k_{i}}^{(i)}\right)
\right\vert ^{q}\right)  ^{1/q}\label{t}%
\end{equation}
for all positive integers $n$ and all $x_{k}^{(i)}\in E_{i}$, with $1\leq
k\leq n$ and$\,1\leq i\leq m$. The vector space of all multiple $(p,q)$%
-summing operators is denoted by $\Pi_{(p,q)}^{mult}\left(  E_{1},\dots
,E_{m};F\right)  $. The infimum $\pi_{\left(  p,q\right)  }^{mult}(T)$, taken
over all possible constants $C$ satisfying (\ref{t}), defines a complete norm
in $\Pi_{\left(  p,q\right)  }^{mult}\left(  E_{1},\dots,E_{m};F\right)  $.
When $E_{1}=\cdots=E_{m}=E$, we write $\Pi_{(p,q)}^{mult}\left(
^{m}E;F\right)  $. For details on the theory of multilinear operators on
Banach spaces we refer to \cite{dineen, mujica} and for more details on
multiple summing operators we refer to \cite{Pop, Pop1} and references therein.

If an $m$-linear operator $T:E_{1}\times\cdots\times E_{m}\rightarrow F$ does
not satisfy (\ref{t}) it makes sense to try to have a measure of how far $T$
is from satisfying (\ref{t}); in this line, in \cite{MPS} the notion of index
of summability is considered.

\begin{definition}
\label{mariana} The \textit{multilinear} $m$-\textit{index of} $\left(
p,q\right)  $-\textit{summability} of a pair $\left(  E_{1}\times\cdots\times
E_{m},F\right)  $ is defined as
\[
\mathbb{\eta}_{(p,q)}^{m-mult}\left(  E_{1},...,E_{m};F\right)  =\inf
s_{m,p,q},
\]
where $s_{m,p,q}\geq0$ satisfies the following:

There is a constant $C\geq0$ (not depending on $n$) such that
\[
\left(  \sum_{k_{1},...,k_{m}=1}^{n}\left\Vert T\left(  x_{k_{1}}%
^{(1)},...,x_{k_{m}}^{(m)}\right)  \right\Vert ^{p}\right)  ^{\frac{1}{p}}\leq
Cn^{s_{m,p,q}}\prod_{i=1}^{m}\left(  \sup_{\varphi_{i}\in B_{E_{i}^{\ast}}%
}\sum\limits_{k_{i}=1}^{n}\left\vert \varphi_{i}\left(  x_{k_{i}}%
^{(i)}\right)  \right\vert ^{q}\right)  ^{1/q}%
\]
for every $T\in\mathcal{L}(E_{1},...,E_{m};F),$ all positive integers $n$ and
$x_{k_{i}}^{(i)}\in E_{i}$, with $1\leq k_{i}\leq n\text{ and }1\leq i\leq m$.

When $E_{1}=\cdots=E_{m}=E$, we write $\mathbb{\eta}_{(p,q)}^{m-mult}\left(
^{m} E;F\right)  $ instead of $\mathbb{\eta}_{(p,q)}^{m-mult}\left(
E,...,E;F\right)  $.
\end{definition}

In some sense this index measures how distant are the spaces $\Pi
_{(p,q)}^{mult}\left(  E_{1},\dots,E_{m};F\right)  $ and the space of
continuous $m$-linear operators from $E_{1}\times\cdots\times E_{m}$ to $F$,
denoted by $\mathcal{L}\left(  E_{1},\dots,E_{m};F\right)  $ . When both
spaces coincide we have
\[
\mathbb{\eta}_{(p,q)}^{m-mult}\left(  E_{1},...,E_{m};F\right)  =0.
\]
For polynomials the situation is similar. Let $\mathcal{P}(^{m}E;F)$ denote
the Banach space of all continuous $m$-homogeneous polynomials from $E$ into
$F$. We recall that given $1\leq p,q<\infty,$ with $p\geq\frac{q}{m},$ a
polynomial $P\in\mathcal{P}(^{m}E;F)$ is \textit{absolutely} $\left(
p,q\right)  $\textit{-summing} if there is a constant $C\geq0$ such that%
\begin{equation}
\left(  \sum_{k=1}^{n}\left\Vert P(x_{k})\right\Vert ^{p}\right)  ^{\frac
{1}{p}}\leq C\left(  \sup_{\varphi_{i}\in B_{E^{\ast}}}\sum\limits_{k=1}%
^{n}\left\vert \varphi_{i}\left(  x_{k}\right)  \right\vert ^{q}\right)
^{m/q}\label{8866}%
\end{equation}
for all positive integers $n$ and all $x_{k}\in E$, with $1\leq k\leq n$. We
denote by $\mathcal{P}_{(p,q)}\left(  ^{m}E;F\right)  $ the Banach space of
all absolutely $\left(  p,q\right)  $-summing polynomials from $E$ to $F$. We
refer the interested reader to \cite{AlbBayPelS, Bot, RuePer} for more recent
results and further details.

\begin{definition}
The \textit{polynomial} $m$-\textit{index of} $\left(  p,q\right)
$-\textit{summability} of a pair of Banach spaces $\left(  E,F\right)  $ is
defined as
\[
\mathbb{\eta}_{(p,q)}^{m-pol}\left(  E,F\right)  =\inf s_{m,p,q},
\]
where $s_{m,p,q}\geq0$ satisfies the following:

There is a constant $C\geq0$ (not depending on $n$) such that
\[
\left(  \sum_{j=1}^{n}\left\Vert P(x_{j})\right\Vert ^{p}\right)  ^{\frac
{1}{p}}\leq Cn^{s_{m,p,q}}\left(  \sup_{\varphi_{i}\in B_{E^{\ast}}}%
\sum\limits_{k=1}^{n}\left\vert \varphi_{i}\left(  x_{k}\right)  \right\vert
^{q}\right)  ^{m/q}%
\]
for every $P\in\mathcal{P}(^{m}E;F)$, all positive integers $n$ and all
$x_{j}\in E$, with $1\leq j\leq n.$
\end{definition}

For the sake of simplicity, given $x_{1},...,x_{n}\in E$ we shall, as usual,
denote
\[
\Vert(x_{k})_{k=1}^{n}\Vert_{w,p}:=\sup_{\varphi\in B_{E^{\ast}}}\left(
\sum_{k=1}^{n}\left\vert \varphi(x_{k})\right\vert ^{p}\right)  ^{\frac{1}{p}%
}.
\]

From now on, for $p \in\left[  1, + \infty\right]  $, $p^{\ast}$ denotes the
conjugate of $p$, i.e., $\frac{1} {p}+\frac{1}{p^{\ast}}=1$. For $1 \leq p <
\infty,$ let us set $X_{p} : = \ell_{p}$ and let us define $X_{\infty}=c_{0}.$

\bigskip{}Every result of the form
\[
\mathcal{L}\left(  E_{1},... ,E_{m};F\right)  =\Pi_{t,s}^{mult}\left(
E_{1},... ,E_{m};F\right)
\]
is called coincidence situation (see, for instance, \cite{BotBraJunPel,
BotMicPel}). It is not difficult to show how coincidence situations can be
used to estimate the index of summability.

\begin{proposition}
\label{estimativaporcoincidenciamult} Let $E_{1},... ,E_{m},F$ be Banach
spaces. Suppose that
\[
\mathcal{L}\left(  E_{1},... ,E_{m};F\right)  =\Pi_{t,s}^{mult}\left(
E_{1},... ,E_{m};F\right)  .
\]
Then

\begin{enumerate}
\item[(a)] For all $p,q$ satisfying $0<p\leq t$ and $0<s\leq q$, we have
\[
\eta_{(p,q)}^{m-mult}\left(  E_{1},... ,E_{m};F\right)  \leq\frac{m}{p}%
-\frac{m}{t}+\frac{m}{s}-\frac{m}{q}.
\]

\item[(b)] For all $p,q$ satisfying $0<p\leq t$ and $0<q\leq s$, we have
\[
\eta_{(p,q)}^{m-mult}\left(  E_{1},... ,E_{m};F\right)  \leq\frac{m}{p}%
-\frac{m}{t}.
\]

\item[(c)] For all $p,q$ satisfying $0<t\leq p$ and $0<s\leq q$, we have
\[
\eta_{(p,q)}^{m-mult}\left(  E_{1},... ,E_{m};F\right)  \leq\frac{m}{s}%
-\frac{m}{q}.
\]

\item[(d)] For all $p,q$ satisfying $0<t\leq p$ and $0<q\leq s$, we have
\[
\eta_{(p,q)}^{m-mult}\left(  E_{1},... ,E_{m};F\right)  =0.
\]

\end{enumerate}
\end{proposition}

\begin{proof}
Let $T\in\mathcal{L}(E_{1},...,E_{m};F).$

(a) If $p\leq t,$ then
\begin{align*}
\left(  \sum_{k_{1},...,k_{m}=1}^{n}\Vert T(x_{k_{1}}^{(1)},...,x_{k_{m}%
}^{(m)})\Vert^{p}\right)  ^{\frac{1}{p}} &  \leq\left(  \sum_{k_{1}%
,...,k_{m}=1}^{n}\Vert T(x_{k_{1}}^{(1)},...,x_{k_{m}}^{(m)})\Vert^{t}\right)
^{\frac{1}{t}}\left(  \sum_{k_{1},...,k_{m}=1}^{n}|1|^{\frac{pt}{t-p}}\right)
^{\frac{1}{p}-\frac{1}{t}}\\
&  \leq C\prod_{i=1}^{m}\Vert(x_{k_{i}}^{(i)})_{k_{i}=1}^{n}\Vert_{w,s}\left(
n^{m}\right)  ^{\frac{1}{p}-\frac{1}{t}}\\
&  \leq Cn^{\frac{m}{p}-\frac{m}{t}}\prod_{i=1}^{m}\Vert(x_{k_{i}}%
^{(i)})_{k_{i}=1}^{n}\Vert_{w,s}.
\end{align*}

Since $s\leq q$, using the H\"{o}lder inequality for $i=1,... ,m$, we have
\begin{align*}
\Vert(x_{k_{i}}^{(i)})_{k_{i}=1}^{n}\Vert_{w,s} &  =\sup_{\varphi\in
B_{E_{i}^{\ast}}}\left(  \sum_{k_{i}=1}^{n}|\varphi(x_{k_{i}}^{(i)}%
)|^{s}\right)  ^{\frac{1}{s}}\\
&  \leq\sup_{\varphi\in B_{E_{i}^{\ast}}}\left[  \left(  \sum_{k_{i}=1}%
^{n}|\varphi(x_{k_{i}}^{(i)})|^{q}\right)  ^{\frac{1}{q}}\left(  \sum
_{k_{i}=1}^{n}|1|^{\frac{qs}{q-s}}\right)  ^{\frac{1}{s}-\frac{1}{q}}\right]
\\
&  \leq n^{\frac{1}{s}-\frac{1}{q}}\Vert(x_{k_{i}}^{(i)})_{k_{i}=1}^{n}%
\Vert_{w,q}.
\end{align*}
Thus
\[
\left(  \sum_{k_{1},...,k_{m}=1}^{n}\Vert T(x_{k_{1}}^{(1)},...,x_{k_{m}%
}^{(m)})\Vert^{p}\right)  ^{\frac{1}{p}}\leq Cn^{\frac{m}{p}-\frac{m}{t}%
+\frac{m}{s}-\frac{m}{q}}\prod_{i=1}^{m}\Vert(x_{k_{i}}^{(i)})_{k_{i}=1}%
^{n}\Vert_{w,q}.
\]

(b) If $p\leq t$ and $q\leq s,$ we just need to use the canonical inclusion
between the $\ell_{q}^{w}$ spaces to obtain%
\[
\left(  \sum_{k_{1},...,k_{m}=1}^{n}\Vert T(x_{k_{1}}^{(1)},...,x_{k_{m}%
}^{(m)})\Vert^{p}\right)  ^{\frac{1}{p}}\leq Cn^{\frac{m}{p}-\frac{m}{t}}%
\prod_{i=1}^{m}\Vert(x_{k_{i}}^{(i)})_{k_{i}=1}^{n}\Vert_{w,q}.
\]

The cases (c) and (d) are similar.
\end{proof}

In the context of polynomials the above result is translated as:

\begin{proposition}
\label{estimativaporcoincidenciapol} Let $E,F$ be Banach spaces and
\[
\mathcal{P}\left(  ^{m}E;F\right)  =\mathcal{P}_{t,s}\left(  ^{m}E;F\right)  .
\]
Then

\begin{enumerate}
\item[(a)] For all $p,q$ satisfying $0<p\leq t$ and $0<s\leq q$, we have
\[
\eta_{(p,q)}^{m-pol}\left(  E;F\right)  \leq\frac{1}{p}-\frac{1}{t}+\frac
{m}{s}-\frac{m}{q}.
\]

\item[(b)] For all $p,q$ satisfying $0<p\leq t$ and $0<q\leq s$, we have
\[
\eta_{(p,q)}^{m-pol}\left(  E;F\right)  \leq\frac{1}{p}-\frac{1}{t}.
\]

\item[(c)] For all $p,q$ satisfying $0<t\leq p$ and $0<s\leq q$, we have
\[
\eta_{(p,q)}^{m-pol}\left(  E;F\right)  \leq\frac{m}{s}-\frac{m}{q}.
\]

\item[(d)] For all $p,q$ satisfying $0<t\leq p$ and $0<q\leq s$, we have
\[
\eta_{(p,q)}^{m-pol}\left(  E;F\right)  =0.
\]

\end{enumerate}
\end{proposition}

\section{Some upper estimates for the index of summability}

The following lemma is quite useful in the theory of summing operators (see
\cite[Corollary 3.20]{DavVil} or \cite{dimant}).

\begin{lemma}
\label{truque} Let $p_{1},...,p_{m}\in\lbrack1,\infty]$. Then
\[
\mathcal{L}(E_{1},...,E_{m};F)=\Pi_{(t,p_{1}^{\ast},...,p_{m}^{\ast})}%
^{mult}(E_{1},...,E_{m};F)
\]
for all Banach spaces $E_{1},...,E_{m}$ if and only if there is a constant
$C>0$ such that
\[
\left(  \sum_{k_{1},...,k_{m}=1}^{n}\Vert T(e_{k_{1}},...,e_{k_{m}})\Vert
^{t}\right)  ^{\frac{1}{t}}\leq C\Vert T\Vert
\]
for every $m$-linear form $T:\ell_{p_{1}}^{n}\times\cdots\times\ell_{p_{m}%
}^{n}\rightarrow F.$
\end{lemma}

We recall that for $2\leq q\leq\infty$, a Banach space $E$ has cotype $q$ if
there is a constant $C\geq0$ such that no matter how we select finitely many
vectors $x_{1},...,x_{n}$ from $E$,
\[
\left(  \sum_{k=1}^{n}\left\Vert x_{k}\right\Vert ^{q}\right)  ^{\frac{1}{q}%
}\leq C\left(  \int_{0}^{1}\left\Vert \sum_{k=1}^{n}r_{k}(t)x_{k}\right\Vert
^{2}dt\right)  ^{\frac{1}{2}},
\]
where $r_{k}$ denotes the $k$-th Rademacher function, that is, given
$k\in\mathbb{N}\text{ and }t\in\left[  0,1\right]  ,$ we have $r_{k}%
(t)=\mathrm{sign}\left[  \mathrm{sin}\left(  2^{k}\pi t\right)  \right]  .$
When $q=\infty,$ the left hand side will be replaced by the sup norm.
Henceforth we will denote $\inf\{q:E\mbox{ has cotype }q\}$ by $\cot(E)$.

We shall need the following coincidence theorem that can be essentially found
in \cite{AlbBayPelS, AraPel1} to obtain estimates of the indices of
summabilitity along this paper.

\begin{theorem}
\label{teocoin} Let $E_{1},...,E_{m},F$ be infinite-dimensional Banach spaces
and suppose that $F$ has finite cotype $\cot(F)=r.$

(a) If $s\in\left[  1,2\right)  $ and $m<\frac{s}{r(s-1)}$, then
\[
\mathcal{L}\left(  E_{1},...,E_{m};F\right)  =\Pi_{t,s}^{mult}\left(
E_{1},... ,E_{m};F\right)  \Leftrightarrow t\geq\frac{sr}{s-msr+mr}.
\]

(b) If $t\in\left[  \frac{2m}{m+1},2\right] $, then%
\[
\mathcal{L}\left(  E_{1},... ,E_{m};\mathbb{K}\right)  =\Pi_{t,s}%
^{mult}\left(  E_{1},... ,E_{m};\mathbb{K}\right)  \Leftrightarrow s\leq
\frac{2mt}{mt+2m-t}.
\]

(c) If $t\in\left(  2,\infty\right) $, then%
\[
\mathcal{L}\left(  E_{1},... ,E_{m};\mathbb{K}\right)  =\Pi_{t,s}%
^{mult}\left(  E_{1},... ,E_{m};\mathbb{K}\right)  \Leftrightarrow s\leq
\frac{mt}{mt+1-t}.
\]

\end{theorem}

\begin{proof}
In \cite[Theorem 1.5]{AlbBayPelS} it was proved that if $p_{1},...,p_{m}%
\in\left[  2,\infty\right]  ,$ and $F$ is infinite-dimensional with finite
cotype $\cot\left(  F\right)  :=r$, with $\frac{1}{p_{1}}+\cdots+\frac
{1}{p_{m}}<\frac{1}{r},$ then there is a constant $C_{p_{1},...,p_{m}}\geq1$
such that
\[
\left(  \sum_{k_{1},...,k_{m}=1}^{\infty}\Vert A(e_{k_{1}}^{(1)},...,e_{k_{m}%
}^{(m)})\Vert^{t}\right)  ^{\frac{1}{t}}\leq C_{p_{1},...,p_{m}}\left\Vert
A\right\Vert \Leftrightarrow\frac{1}{t}\leq\frac{1}{r}-\left(  \frac{1}{p_{1}%
}+...+\frac{1}{p_{m}}\right)
\]
for every continuous $m$-linear operator $A:\ X_{p_{1}}\times\cdots\times
X_{p_{m}}\rightarrow F$ (see also \cite{dimant}).

By Lemma \ref{truque}, with $p_{i}=s^{\ast},\text{ for all }i,$ this result is
translated to the language of multiple summing operators and we prove (a).

The proofs of (b) and (c) can be found in \cite[Theorem 3.2]{AraPel1}.
\end{proof}

An immediate corollary of Theorem \ref{teocoin}(a) and Proposition
\ref{estimativaporcoincidenciamult} is the following:

\begin{corollary}
\label{cornbd} Let $E_{1},...,E_{m},F$ be infinite-dimensional Banach spaces.
If $F$ has finite cotype $\cot(F)=r<\infty$, $1\leq s< 2$, $m<\frac{s}%
{r(s-1)}$ and $t=\frac{sr}{s-msr+mr}$, then

\begin{enumerate}
\item[(a)] For all $p,q$ satisfying $0<p\leq t$ and $\frac{mrt}{r-t+mrt}\leq
q$, we have
\[
\eta_{(p,q)}^{m-mult}\left(  E_{1},...,E_{m};F\right)  \leq\frac{m}{p}%
+m-\frac{1}{r}-\frac{m}{q}-\frac{(m-1)}{t}.
\]

\item[(b)] For all $p,q$ satisfying $0<p\leq t$ and $0<q\leq\frac
{mrt}{r-t+mrt}$, we have
\[
\eta_{(p,q)}^{m-mult}\left(  E_{1},...,E_{m};F\right)  \leq\frac{m}{p}%
-\frac{m}{t}.
\]

\item[(c)] For all $p,q$ satisfying $0<t\leq p$ and $\frac{mrt}{r-t+mrt}\leq
q$, we have
\[
\eta_{(p,q)}^{m-mult}\left(  E_{1},...,E_{m};F\right)  \leq m-\frac{1}%
{r}+\frac{1}{t}-\frac{m}{q}.
\]

\item[(d)] For all $p,q$ satisfying $0<t\leq p$ and $0<q\leq\frac
{mrt}{r-t+mrt}$, we have
\[
\eta_{(p,q)}^{m-mult}\left(  E_{1},...,E_{m};F\right)  =0.
\]

\end{enumerate}
\end{corollary}

\bigskip

\section{Optimal estimates for the index of summability}

We begin this section by recalling a general version of the
Kahane--Salem--Zygmund inequality:

\begin{lemma}
\label{ksz}(See Albuquerque et al. \cite[Lemma 6.1]{nacib}) Let $m,n\geq1$,
let $p\in\left[  1,\infty\right]  ,$ and let
\[
\alpha\left(  p\right)  =\left\{
\begin{array}
[c]{ll}%
\frac{1}{2}-\frac{1}{p}, & \text{if }p\geq2\\
0, & \text{otherwise}.
\end{array}
\right.
\]
There is an universal constant $C_{m}$ (depending only on $m$) and there
exists an $m-$linear form $A:\ell_{p}^{n}\times\cdots\times\ell_{p}%
^{n}\rightarrow\mathbb{K}$ of the form
\[
A(z^{(1)},...,z^{(m)})=\displaystyle\sum_{i_{1},...,i_{m}=1}^{n}\pm z_{i_{1}%
}^{(1)}\cdots z_{i_{m}}^{(m)}%
\]
such that
\[
\Vert A\Vert\leq C_{m}n^{\frac{1}{2}+m\cdot\alpha(p)}.
\]

\end{lemma}

Now we can obtain optimal indices of summability for certain pairs of Banach spaces.

\begin{proposition}
Let $p,q$ be real numbers.

(a) If $\frac{2m}{m+1}\leq p\leq2\text{ and }\frac{2mp}{mp+2m-p}\leq q\leq2,$
then
\[
\eta_{(p,q)}^{m-mult}\left(  ^{m}\ell_{q^{\ast}};\mathbb{K}\right)  =\frac
{m}{p}+\frac{m}{2}-\frac{1}{2}-\frac{m}{q}.
\]

(b) If $2<p<\infty\text{ and }\frac{mp}{mp+1-p}\leq q$, then
\[
\eta_{(p,q)}^{m-mult}\left(  ^{m}\ell_{q^{\ast}};\mathbb{K}\right)
=m-1+\frac{1}{p}-\frac{m}{q}.
\]

(c) If $0< p < \infty\text{ and } 1 \leq q \leq2,$ then
\[
\eta_{(p,q)}^{m-mult}\left( ^{m} \ell_{q^{\ast}};c_{0}\right)  =\frac{m}{p}.
\]

\end{proposition}

\begin{proof}
(a) Note that we can obtain an upper estimate using Theorem \ref{teocoin} (b)
and the first item of Proposition \ref{estimativaporcoincidenciamult}. In
fact, Theorem \ref{teocoin} tells us that
\[
\mathcal{L}\left(  E_{1},... ,E_{m};\mathbb{K}\right)  =\Pi_{t,\frac
{2mt}{mt+2m-t}}^{mult}\left(  E_{1},... ,E_{m};\mathbb{K}\right)
\]
and using the first item of Proposition \ref{estimativaporcoincidenciamult},
with $t=p$, we obtain
\[
\eta_{(p,q)}^{m-mult}\left(  E_{1},...,E_{m};\mathbb{K}\right)  \leq\frac
{m}{p}+\frac{m}{2}-\frac{1}{2}-\frac{m}{q}.
\]

Now let us show that this estimate is optimal for $E_{i}=\ell_{q^{\ast}}.$ By
Lemma \ref{ksz}, given $q^{\ast}\geq2$ (thus $q\leq2$) there is an operator
$A:\ell_{q^{\ast}}^{n}\times\cdots\times\ell_{q^{\ast}}^{n}\rightarrow
\mathbb{K}$ given by
\[
A(z^{(1)},...,z^{(m)})=\displaystyle\sum_{i_{1},...,i_{m}=1}^{n}\pm z_{i_{1}%
}^{(1)}\cdots z_{i_{m}}^{(m)}%
\]
such that
\[
\Vert A\Vert\leq C_{m}n^{\frac{1}{2}+m\left(  \frac{1}{2}-\frac{1}{q^{\ast}%
}\right)  }.
\]
Suppose that
\[
\left(  \sum_{k_{1},...,k_{m}=1}^{n}\Vert A(e_{k_{1}}^{(1)},...,e_{k_{m}%
}^{(m)})\Vert^{p}\right)  ^{\frac{1}{p}}\leq Cn^{s}\Vert A\Vert.
\]
Then
\[
n^{\frac{m}{p}}\leq Cn^{s}n^{\frac{1}{2}+m\left(  \frac{1}{2}-\frac{1}%
{q^{\ast}}\right)  }%
\]
and since $n$ is arbitrary,
\[
\frac{m}{p}-\frac{1}{2}-\frac{m}{2}+\frac{m}{q^{\ast}}\leq s.
\]
Thus
\[
\eta_{(p,q)}^{m-mult}\left(  ^{m}\ell_{q^{\ast}};\mathbb{K}\right)  \geq
\frac{m}{p}-\frac{1}{2}-\frac{m}{2}+\frac{m}{q^{\ast}}=\frac{m}{p}-\frac{1}%
{2}+\frac{m}{2}-\frac{m}{q}%
\]
and this proves (a).

(b) Using item (c) of Theorem \ref{teocoin} and the first item of Proposition
\ref{estimativaporcoincidenciamult}, we have
\[
\mathcal{L}\left(  E_{1},... ,E_{m};\mathbb{K}\right)  =\Pi_{t,\frac
{mt}{mt+1-t}}^{mult}\left(  E_{1},... ,E_{m};\mathbb{K}\right)
\]
and considering $t=p$ we obtain
\[
\eta_{(p,q)}^{m-mult}\left(  E_{1},...,E_{m};\mathbb{K}\right)  \leq
m+\frac{1}{p}-1-\frac{m}{q}.
\]
Let us show that the estimate is sharp for $E_{i}=\ell_{q^{\ast}}.$ Consider
$S:\ell_{q^{\ast}}^{n}\times\cdots\times\ell_{q^{\ast}}^{n}\rightarrow
\mathbb{K}$ given by $S(x^{(1)},...,x^{(m)})=\displaystyle\sum_{i=1}^{n}%
x_{i}^{(1)}\cdots x_{i}^{(m)}$ and note that $\Vert S\Vert\leq n^{1-\frac
{m}{q^{\ast}}}.$ If
\[
\left(  \sum_{k_{1},...,k_{m}=1}^{n}\Vert S(e_{k_{1}}^{(1)},...,e_{k_{m}%
}^{(m)})\Vert^{p}\right)  ^{\frac{1}{p}}\leq Cn^{s}\Vert S\Vert,
\]
then
\[
n^{\frac{1}{p}}\leq Cn^{s}n^{1-\frac{m}{q^{\ast}}}.
\]
Therefore
\[
\frac{1}{p}-1+\frac{m}{q^{\ast}}\leq s
\]
and
\[
\eta_{(p,q)}^{m-mult}\left(  ^{m}\ell_{q^{\ast}};\mathbb{K}\right)  \geq
\frac{1}{p}-1+\frac{m}{q^{\ast}}=\frac{1}{p}+m-1-\frac{m}{q}.
\]

(c) Let $t$ be a positive real number such that for each $T\in\mathcal{L}%
(^{m}\ell_{q^{\ast}};c_{0})$ there is a constant $C\geq0$ such that
\begin{equation}
\left(  \sum_{k_{1},...,k_{m}=1}^{n}\left\Vert T\left(  x_{k_{1}}%
^{(1)},...,x_{k_{m}}^{(m)}\right)  \right\Vert ^{p}\right)  ^{\frac{1}{p}}\leq
Cn^{t}\prod_{i=1}^{m}\left\Vert \left(  x_{k_{i}}^{(i)}\right)  _{k_{i}=1}%
^{n}\right\Vert _{w,q}\label{joma1}%
\end{equation}
for all positive integers $n$ and all $x_{k_{i}}^{(i)}\in\ell_{q^{\ast}}$,
with $1\leq k_{i}\leq n$.

Now, let $T\in\mathcal{L}(^{m}\ell_{q^{\ast}};c_{0})$ be defined by
\[
T\left(  x^{(1)},...,x^{(m)}\right)  =\left(  x_{j_{1}}^{(1)}\cdots x_{j_{m}%
}^{(m)}\right)  _{j_{1},...,j_{m}=1}^{n}.
\]
Of course $\left\Vert T\right\Vert =1$ and
\[
\left(  \sum_{j_{1},...,j_{m}=1}^{n}\left\Vert T(e_{j_{1}},...,e_{j_{m}%
})\right\Vert ^{p}\right)  ^{\frac{1}{p}}=n^{\frac{m}{p}}.
\]
We have $\left\Vert (e_{j_{i}})_{j_{i}=1}^{n}\right\Vert _{w,q}=1,$ the latter
condition together with (\ref{joma1}) imply
\[
n^{\frac{m}{p}}\leq Cn^{t}%
\]
and thus $t\geq\frac{m}{p}$. The reverse inequality follows by
\cite[Proposition 2.6 and 2.7]{MPS}.
\end{proof}

\bigskip

To prove our final results we recall two coincidence theorems.\ The first one
is a result credited to Defant and Voigt (see \cite{AleMat}), and asserts that
for any Banach space $E$ we have
\begin{equation}
\mathcal{P}\left(  ^{m}E;\mathbb{K}\right)  =\mathcal{P}_{1,1}\left(
^{m}E;\mathbb{K}\right)  .\label{ddddd}%
\end{equation}
The second one is a result due to Botelho \cite{Bot} asserting that for any
Banach space $E,F$, if $\cot(E)=mr$ or $cot(F)=r$, then
\begin{equation}
\mathcal{P}\left(  ^{m}E;F\right)  =\mathcal{P}_{r,1}\left(  ^{m}E;F\right)
.\label{bbbb}%
\end{equation}

We also need to recall the main results of \cite{MPS}:

\begin{theorem}
\label{mps} (See Maia, Pellegrino, Santos \cite{MPS}) Let $E,F$ be infinite
dimensional Banach spaces and $r:=\cot\left(  F\right)  .$ Then

\begin{enumerate}
\item[(a)] For $1 \leq q \leq2$ and $0< p \leq\frac{rq}{mr+q}$, we have
\[
\frac{m}{2} \leq\eta_{(p,q)}^{m-pol}\left(  E;F\right) .
\]

\item[(b)] For $1 \leq q \leq2$ and $\frac{rq}{mr+q} \leq p \leq\frac
{2r}{mr+2}$, we have
\[
\frac{mp+2}{2p}-\frac{mr+q}{rq} \leq\eta_{(p,q)}^{m-pol}\left(  E;F\right) .
\]

\item[(c)] For $2 \leq q < \infty\text{ and } 0 < p \leq\frac{2r}{mr+2},$ we
have
\[
\frac{m}{2}\leq\eta_{(p,q)}^{m-pol}\left(  E;F\right) .
\]

\item[(d)] For $2 \leq q < \infty\text{ and } \frac{2r}{mr+2} < p < r,$ we
have
\[
\frac{r-p}{pr}\leq\eta_{(p,q)}^{m-pol}\left(  E;F\right) .
\]

\end{enumerate}
\end{theorem}

Note that there is a range, $1 \leq q \leq2 \text{ and } \frac{2r}{mr+2} < p <
r,$ where the previous theorem doesn't provide any information, as illustrated below.

\begin{figure}[h]
\centering
\includegraphics[height=4cm]{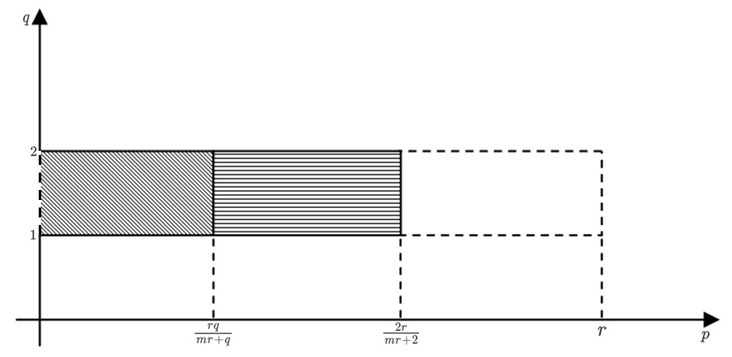}\caption{Region encompassed by (a)
and (b) of Theorem \ref{mps}}%
\end{figure}

\begin{figure}[h]
\centering
\includegraphics[height=5cm]{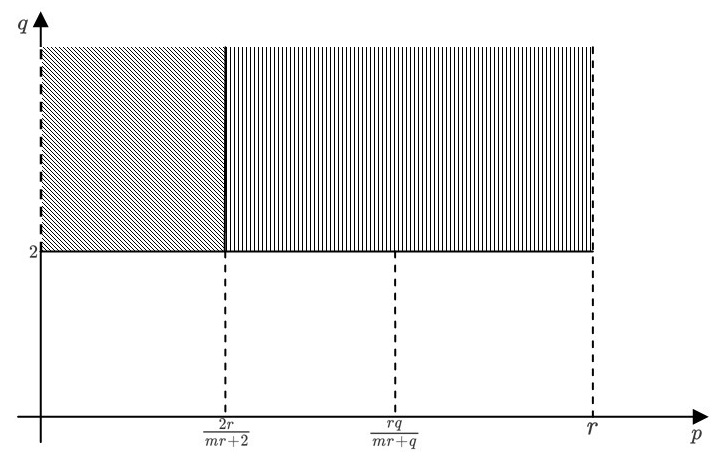}\caption{Region encompassed by (c)
and (d) of Theorem \ref{mps}}%
\end{figure}

\bigskip

The next theorem shows that it is possible to consider $q$ varying in a larger
interval than the interval considered in $(d)$.

\begin{theorem}
\label{cotipon} Let $E,F$ be infinite dimensional Banach spaces and
$r:=\cot\left(  F\right)  .$
\end{theorem}

For $1\leq q<\infty\text{ and }\frac{2r}{mr+2}<p<r,$ we have
\[
\frac{r-p}{pr}\leq\eta_{(p,q)}^{m-pol}\left(  E;F\right)  .
\]

\begin{proof}
Suppose that%
\[
\left(  \sum\limits_{j=1}^{n}\left\Vert P_{n}\left(  x_{j}\right)  \right\Vert
^{p}\right)  ^{\frac{1}{p}}\leq Dn^{t}\left\Vert P_{n}\right\Vert \left\Vert
\left(  x_{j}\right)  _{j=1}^{n}\right\Vert _{w,q}^{m},
\]
for all $x_{1},...,x_{n} \in E.$

Then, using the argument from the proof of \cite{MPS}, we have
\begin{equation}
\frac{\left(  \sum\limits_{j=1}^{n}\left\Vert id_{X}(x_{j})\right\Vert
^{mp\left(  \frac{r}{p}\right)  ^{\ast}}\right)  ^{\frac{1}{mp\left(  \frac
{r}{p}\right)  ^{\ast}}}}{\left\Vert (x_{j})_{j=1}^{n}\right\Vert _{w,q}}\leq
n^{\frac{t}{m}}Q^{\frac{1}{m}},\label{yyy}%
\end{equation}
for any $n$-dimensional subspace $X$ of $E$ and for all nonzero $x_{1}%
,...,x_{n} \in X.$

Since $q\geq1,$ we have
\[
\frac{\left(  \sum\limits_{j=1}^{n}\left\Vert id_{X}(x_{j})\right\Vert
^{mp\left(  \frac{r}{p}\right)  ^{\ast}}\right)  ^{\frac{1}{mp\left(  \frac
{r}{p}\right)  ^{\ast}}}}{\left\Vert (x_{j})_{j=1}^{n}\right\Vert _{w,1}}\leq
n^{\frac{t}{m}}Q^{\frac{1}{m}}.
\]
for all $x_{1},...,x_{n}\in X$. Then
\[
\pi_{mp\left(  \frac{r}{p}\right)  ^{\ast},1}^{(n)}(id_{X})\leq n^{\frac{t}%
{m}}Q^{\frac{1}{m}}.
\]
Since $\frac{2r}{mr+2}<p$ we have $mp\left(  \frac{r}{p}\right)  ^{\ast}>2$ ,
and from Theorem \cite[Corollary 2]{ko} there is a constant $c$, depending on
$mp\left(  \frac{r}{p}\right)  ^{\ast}$, such that
\begin{equation}
\frac{1}{c}\pi_{mp\left(  \frac{r}{p}\right)  ^{\ast},1}(id_{X})\leq
n^{\frac{t}{m}}Q^{\frac{1}{m}}.\label{jj}%
\end{equation}
By \cite[Proposition 3]{mau}, the $\left(  mp\left(  \frac{r}{p}\right)
^{\ast},1\right)  $- and $\left(  mp\left(  \frac{r}{p}\right)  ^{\ast
},2\right)  $-absolutely summing norms are equivalent, then there exists a
constant $k$ such that
\[
\pi_{mp\left(  \frac{r}{p}\right)  ^{\ast},2}\left(  id_{X}\right)  \leq
k\pi_{mp\left(  \frac{r}{p}\right)  ^{\ast},1}(id_{X}).
\]
Hence
\[
\frac{1}{ck}\pi_{mp\left(  \frac{r}{p}\right)  ^{\ast},2}\left(
id_{X}\right)  \leq n^{\frac{t}{m}}Q^{\frac{1}{m}}.
\]
By \cite[Corollary 2(a)]{Konig}, there is a constant $A>0$ such that
\[
A\cdot n^{\frac{1}{mp\left(  \frac{r}{p}\right)  ^{\ast}}}\leq\pi_{mp\left(
\frac{r}{p}\right)  ^{\ast},2}(id_{X}),
\]
and thus
\[
\frac{A}{ck}n^{\frac{r-p}{mpr}}\leq n^{\frac{t}{m}}Q^{\frac{1}{m}}.
\]
Finally, we obtain
\[
t\geq\frac{r-p}{pr},
\]
and
\[
\eta_{(p,q)}^{m-pol}\left(  E;F\right)  \geq\frac{r-p}{pr}.
\]

\end{proof}

\bigskip

We thus have the following corollary:

\begin{corollary}
Let $E,F$ be infinite dimensional Banach spaces and suppose that $F$ has
cotype $\cot(F):=r<\infty$. Then, for all $\frac{2r}{mr+2}<p<r,$
\[
\eta_{(p,1)}^{m-pol}\left(  E;F\right)  =\frac{1}{p}-\frac{1}{r}.
\]

\end{corollary}

\begin{proof}
To obtain the upper bound we use (a) of Proposition
\ref{estimativaporcoincidenciapol} and (\ref{bbbb}).
\end{proof}

\bigskip

A similar result for $F=\mathbb{R}\text{ and }m\text{ even was proved in the
same paper}$.

The proof of the next result is similar to the proof of the previous theorem,
and we have the following extension of \cite[Theorem 4.1]{MPS}:

\begin{theorem}
\label{realcommpar} Let $\ m\ $be an even positive integer and $E$ be an
infinite dimensional real Banach space. If $1\leq q<\infty\text{ and }\frac
{2}{m+2}<p<1,$ then
\[
\frac{1-p}{p}\leq\eta_{(p,q)}^{m\text{-}pol}\left(  E;\mathbb{R}\right)  .
\]

\end{theorem}

Combining the above theorem and (\ref{ddddd}) we get the following corollary:

\begin{corollary}
Let $E$ be a Banach space and $m$ be an even integer. Then, for $\frac{2}%
{m+2}<p<1$,
\[
\eta_{(p,1)}^{m-pol}\left(  E;\mathbb{R}\right)  =\frac{1}{p}-1.
\]

\end{corollary}

\end{document}